\documentclass[aop]{imsart}

\usepackage{amsmath,amssymb,amsthm}
\usepackage{mathrsfs}
\usepackage{graphicx}
\usepackage{tikz,forest}
\usepackage{dsfont}

\startlocaldefs
\theoremstyle{plain}
\newtheorem{theo}{Theorem}

\newtheorem{prop}{Proposition}
\newtheorem{coro}{Corollary}

\theoremstyle{definition}
\newtheorem{defi}{Definition}

\theoremstyle{definition}
\newtheorem{remark}{Remark}
\endlocaldefs

\begin{document}

\begin{frontmatter}

\title{Reduction and classification of higher-order Markov chains}

\author[A]{\fnms{C.}~\snm{Gallesco}\thanks{email: \texttt{gallesco@unicamp.br}}}
\author[A]{\fnms{C. T.}~\snm{Genovese Huss Oliveira} \thanks{email: \texttt{c247005@dac.unicamp.br}}}
\author[B]{\fnms{D. Y.}~\snm{Takahashi} \thanks{email: \texttt{takahashiyd@gmail.com}}}

\address[A]{Universidade Estadual de Campinas}
\address[B]{Universidade Federal do Rio Grande do Norte}

\begin{abstract}
We study the class structure of finite-alphabet Markov chains with arbitrary memory length. To capture the structural constraints induced by prohibited transitions, we introduce the skeleton of a higher-order transition kernel, defined as a reduced set of contexts encoding all essential zero-probability patterns. To each skeleton we associate a binary transition matrix. We show that the communicating class structure of this matrix completely determines the recurrent classes of the original higher-order Markov chain, along with their periods. As a consequence, simple criteria for essential irreducibility and periodicity follow directly from the skeleton, without constructing the full first-order representation on the enlarged state space. From a practical perspective, this approach can yield significant computational gains. An example illustrates how the skeleton may have substantially smaller order than the original chain.
\end{abstract}


\end{frontmatter}

\section{Introduction}

Markov chains provide a fundamental framework for modeling dependence in categorical time series and play a central role in probability theory and its applications \cite{chung1967markov, norris1998markov}. While the classical theory of Markov chains is largely developed for first-order chains, $m$-th order Markov chains naturally arise in contexts where longer memory effects cannot be neglected, including symbolic dynamics, information theory, and time-series models \cite{fokianos2003regression, ost2023sparse, gallesco2022mixing}. Several modern machine learning algorithms are instances of higher-order Markov chains \cite{zekri2024large}.

A standard approach to studying an $m$-th order Markov chain on a finite alphabet $A$ consists of rewriting it as a first-order Markov chain on the product space $A^m$, see for example \cite{norris1998markov}. Although this reduction is often adequate, it does not explore the sparseness and the special structure of the reduced chain, which contains a large number of zero-probability transitions that do not correspond to genuine prohibitions in the original dynamics, but rather to artifacts of the embedding into $A^m$. Failing to explore the structural properties of the reduced chain can make the analysis of irreducibility, recurrence, and periodicity more obscure and computationally intensive, especially for high-order Markov chains, which are standard in modern applications \cite{ost2023sparse}. 

These difficulties motivate the development of tools that allow one to analyze higher-order Markov chains more intrinsically. In this work, we introduce a reduction framework based on a structural object associated with the transition kernel, which we call the \emph{skeleton}. The skeleton captures precisely the minimal contexts responsible for prohibiting transitions and encodes the essential connectivity structure of the chain.

Our first main result shows that the classification of an $m$-th order Markov chain, \emph{i.e.}, its decomposition into recurrent classes and their periods, is entirely determined by the class structure of a binary matrix naturally associated with the skeleton. In particular, closed classes and their periods for the skeleton matrix correspond exactly to recurrent classes and their periods in the original chain. This establishes the skeleton as a complete invariant for classification purposes.

Beyond its theoretical significance, the skeleton can lead to substantial computational gains. We introduce an explicit and efficient algorithm for extracting the skeleton from a given transition kernel, based on a pruning procedure on a tree representing the transition probabilities of the Markov chain. In several situations, this algorithm enables the effective analysis of chains with large memory requirements. We further derive simple criteria for the irreducibility and essential irreducibility of a Markov chain directly from its skeleton, which, in some instances, allow one to bypass the computation of powers of the transition matrix.

We illustrate the proposed framework on a 10-th order Markov chain example, where the skeleton has significantly smaller order. This example highlights both the conceptual clarity and the computational advantages of the skeleton-based approach.

Finally, in the same vein as this article, we recently introduced a criterion for the uniqueness of chains of infinite order when prohibited transitions exist, extending some of the concepts discussed here to non-Markovian chains \cite{gallesco2025uniqueness}. However, the notion of the skeleton and the complete classification of recurrent classes were not addressed in \cite{gallesco2025uniqueness} and do not appear to generalize easily to the non-Markovian setting.

\section{Theoretical results}
Let $A$ be a finite set, called the {\it alphabet}, and $|A|$ the cardinal of $A$. For ${\bf x}\in A^k$, $k\geq 1$, we denote by $x_i$ the $i$-th coordinate of ${\bf x}$ and for $ 1\leq i \leq j\leq k$ we write ${\bf x}^{j}_{i}:=(x_{i},\ldots, x_{j})$. For ${\bf x}\in A^k$ and ${\bf y} \in A^l$, the \emph{concatenation} ${\bf xy}$ is the new vector ${\bf z}\in A^{k+l}$ such that ${\bf z}_1^k={\bf x}_1^k$ and ${\bf z}_{k+1}^{k+l} = {\bf y}_1^l$. Let ${\bf e}$ be the neutral element of the concatenation operation, that is, ${\bf e}{\bf x}={\bf x}{\bf e}={\bf x}$ for all ${\bf x}$. Throughout the paper we will use the following conventions: if $i>j$, ${\bf x}^{j}_{i}={\bf e}$ and $A^0=\{{\bf e}\}$. Finally, for ${\bf x}\in A^k$, we denote by $|{\bf x}|$ the length of ${\bf x}$, that is, $k$.

We now consider a time-homogeneous Markov chain of order $m\geq 1$, $X=(X_n)_{n\geq 0}$, with state space $A$. We denote by $p$ its transition kernel, that is,
$p:A^m\times A\to [0,1]$ is defined by
\[
p({\bf x},a)=P(X_{n+1}=a\mid X_n=x_m,X_{n-1}=x_{m-1}, \dots,X_{n-m+1}=x_1)
\]
for all ${\bf x}=(x_1,\dots,x_m)$ and $n\geq m-1$.
Equivalently, $(X_n)_{n\geq 0}$ can be seen as a Markov chain of order 1 on the state space $A^m$. In this case, we will denote by $\mathbb{P}$ its transition matrix.

\medskip

A central object in our analysis is the \emph{skeleton}, which encodes the minimal contextual information needed to determine whether a given transition is prohibited or not. To define it precisely, we first introduce, for each pair $({\bf x},a)\in A^m\times A$, the quantity $\tau({\bf x},a)$, which identifies the length of the shortest suffix of ${\bf x}$ that already determines the transition behavior toward $a$.

\begin{defi}
For all $a\in A$ and ${\bf x}\in A^m$, we define 
\[\tau({\bf x},a)=\min\Big\{i\geq 0: p({\bf y}{\bf x}_{m-i}^m,a)=0, \forall {\bf y}\in A^{m-i-1} \; \text{or}\;\; p({\bf y}{\bf x}_{m-i}^m,a)>0,  \forall {\bf y}\in A^{m-i-1}\Big\}
\]
and 
\[
\tau_{{\bf x}}=\sup_{a\in A} \tau({\bf x},a).
\]
The set $\mathcal{S}:=\bigcup_{{\bf x}\in  A^m} \{{\bf x}^m_{m-\tau_{{\bf x}}}\}$ is called the {\it skeleton} of the kernel $p$. Let $K:=1+\sup_{{\bf x}}\tau_{{\bf x}}\in \{1,\dots,m\}$. $K$ is called the {\it order} of the skeleton.
\end{defi}

Intuitively, $\tau({\bf x},a)$ measures how far back in the past one needs to look in order to determine whether the transition to $a$ is forbidden or allowed from the context ${\bf x}$. More precisely, $\tau({\bf x},a)=i$ means that the suffix ${\bf x}_{m-i}^m$ of length $i+1$ is already sufficient to resolve this question: either the transition to $a$ is forbidden regardless of any preceding context ${\bf y}\in A^{m-i-1}$ (i.e., $p({\bf y}{\bf x}_{m-i}^m,a)=0$ for all ${\bf y}$), or it is allowed regardless of any preceding context (i.e., $p({\bf y}{\bf x}_{m-i}^m,a)>0$ for all ${\bf y}$). In other words, no additional past information beyond the last $i+1$ symbols is needed to determine the transition behavior toward $a$.
The quantity $\tau_{{\bf x}} = \sup_{a \in A} \tau({\bf x},a)$ then represents the worst case over all possible target symbols $a\in A$: it is the length of the shortest suffix of ${\bf x}$ that simultaneously resolves the transition behavior toward \emph{every} symbol in the alphabet. Intuitively, $\tau_{{\bf x}}$ is the minimal memory depth required at state ${\bf x}$ to fully characterize all its outgoing transitions. The skeleton $\mathcal{S}$ collects these minimal sufficient suffixes across all states ${\bf x}\in A^m$, and thereby encodes the complete transition structure of the chain in the most compressed form.
\medskip

In the following, we will consider binary square matrices with entries in a finite set $E$, that is, matrices whose elements belongs to $\{0,1\}$. In this case, the addition and product will always be the boolean ones. For example, $1+1=1$, $0\times 1=0$, etc. With these operations, the set of binary square matrices of a given dimension is a boolean algebra. Given a binary square matrix $B$ with entries in $E$, we denote by $B^n(k,l)$ the $(k,l)$-entry of its $n$-th boolean power $B^n$.
We can naturally transpose the notions of communicating class, closed (communicating) class and period of a closed (communicating) class from stochastic matrices to binary square matrices. For example, considering a binary square matrix $B$, a communicating class $\mathcal C\subset E$ is closed if, for all $k\in \mathcal C$, $B^n(k,l)=0$ for all $n\geq 1$ and $l\notin \mathcal C$. For a stochastic matrix, a closed class is also called {\it recurrent}. As for stochastic matrices, binary square matrices induce the following kind of partition of $E$: for each square binary matrix $B$,
there exist $N\geq 1$, $\mathcal{C}_1$, $\dots$, $\mathcal{C}_N$ closed classes and $\mathcal T$ which does not contain any closed class and is disjoint from $\mathcal{C}_1\cup\dots\cup \mathcal{C}_N$ such that 
$E=\mathcal{C}_1\cup \dots\cup \mathcal{C}_N\cup \mathcal{T}$ ($\mathcal{T}$ may be empty).
\begin{defi}
Consider a kernel $p$ with skeleton of order $K$. The {\it skeleton matrix} $\mathbb{M}$ is  the binary matrix of dimension $A^K\times A^K$  defined, for $a\in A$ and ${\bf x}\in A^{K}$, by 
$$\mathbb{M}({\bf u},{\bf v})
=\left\{
\begin{array}{ll}
1, &\mbox{if}\phantom{*}({\bf u},{\bf v})=({\bf x}_{1}^{K},{\bf x}_{2}^{K}a)\;\;\text{and}\;\; p({\bf y}{\bf x}_{1}^{K},a)>0,\; \text{for all}\;\; {\bf y}\in A^{m-K};\\
0, &\mbox{ otherwise}.
\end{array}
\right.
$$
\end{defi}
If $K<m$ we also define
\begin{equation*}
\mathcal{A}=\Big\{{\bf x}\in A^m \; \text{such that}\;\; \mathbb{M}({\bf x}_1^K,{\bf x}_2^{K+1})\mathbb{M}({\bf x}_2^{K+1},{\bf x}_3^{K+2})\dots \mathbb{M}({\bf x}_{m-K}^{m-1},{\bf x}_{m-K+1}^m)=1\Big\}
\end{equation*}
and let $\mathcal{A}=A^m$ if $K=m$.
We now state our main result.
\begin{theo}
\label{theo1}
Consider a Markov chain of order $m$, $X$, with skeleton of order $K$. If $\mathbb{M}$ has $N \geq 1$ closed classes $\mathcal{C}_1,\dots,\;\mathcal{C}_N$, then $\mathbb{P}$ has $N$ recurrent classes $\mathcal{R}_1,\dots,\;\mathcal{R}_N$. The classes $\mathcal{R}_i$ satisfy
\begin{equation*}
\mathcal{R}_i=\{{\bf y}{\bf z}\in A^m: {\bf y}\in \mathcal{C}_i \;\text{and}\; {\bf y}{\bf z}\in \mathcal{A}\}
\end{equation*} 
for $i\in\{1,\dots,N\}$.
Moreover, the period  of $\mathcal{C}_i$ is equal to the period of $\mathcal{R}_i$ for $i\in\{1,\dots,N\}$.
\end{theo}
Let us note that Theorem \ref{theo1} does not say anything about the transient states of $X$. In fact, it may happen that the set $\mathcal{T}=\emptyset$ and $X$ has some transient states. Also, $\mathcal{T}$ may be a communicating class and even so $X$ has several transient classes.
\medskip

\noindent
{\it Proof.}
 We first observe that for $i\in \{1,\dots,N\}$, $\mathcal{R}_i$ is a communicating class. Indeed, consider ${\bf c}, {\bf \hat{c}}\in \mathcal{R}_i$. By definition, there exist ${\bf a}, {\bf \hat{a}}\in \mathcal{C}_i$ and ${\bf b}, {\bf \hat{b}}\in A^{m-K}$ such that ${\bf c}={\bf a}{\bf b}$ and ${\bf \hat{c}}={\bf \hat{a}}{\bf \hat{b}}$. Since $\mathcal{C}_i$ is closed and ${\bf \hat{a}}{\bf \hat{b}}\in \mathcal{A}$, we deduce that $\mathbb{P}^n({\bf a}{\bf b}, {\bf \hat{a}}{\bf \hat{b}})>0$ for some $n\geq 1$. Also observe that, by definition, the classes $\mathcal{R}_i$, $i\in \{1,\dots,N\}$ are  recurrent and disjoint. 
 
 To prove the first part of the theorem, it remains to show that $\mathbb{P}$ cannot have more than $N$ recurrent classes. If $(\mathcal{C}_1\cup\dots\cup\mathcal{C}_N)^c=\emptyset$, then either $(\mathcal{R}_1\cup\dots\cup\mathcal{R}_N)^c=\emptyset$ (and the first part of the theorem is proved) or $(\mathcal{R}_1\cup\dots\cup\mathcal{R}_N)^c\neq \emptyset$. In this second case, $(\mathcal{R}_1\cup\dots\cup\mathcal{R}_N)^c$ contains elements which are not in $\mathcal{A}$ and therefore cannot be recurrent. If $(\mathcal{C}_1\cup\dots\cup\mathcal{C}_N)^c \neq \emptyset$, then either an element of $(\mathcal{R}_1\cup\dots\cup\mathcal{R}_N)^c$ does not belong to $\mathcal{A}$ (and thus cannot be recurrent), or is of the form ${\bf a}{\bf b}$ where ${\bf a}\in (\mathcal{C}_1\cup\dots\cup\mathcal{C}_N)^c$ and ${\bf a}{\bf b}\in \mathcal{A}$. In this second case, there exists $n\geq 1$, $i\in \{1,\dots, N\}$ and ${\bf w}\in \mathcal{C}_i$ such that $\mathbb{P}^n({\bf a}{\bf b}, {\bf w})>0$ and therefore ${\bf a}{\bf b}$ is not recurrent.

Let us denote by $d(\mathcal{C}_i)$ and $ d(\mathcal{R}_i)$ the periods of $\mathcal{C}_i$ and $\mathcal{R}_i$ respectively. To prove the second part of the theorem, first observe that $d(\mathcal{C}_i)\leq d(\mathcal{R}_i)$ by definition of $\mathcal{R}_i$. We will show by contradiction that we cannot have  $d(\mathcal{C}_i)< d(\mathcal{R}_i)$. For the rest of the proof, let us denote $d=d(\mathcal{C}_i)$ and let $\mathcal{D}_0,\dots,\;\mathcal{D}_{d-1}$ be the corresponding cycling classes. Consider ${\bf c}\in \mathcal{R}_i$. We have that ${\bf c}_1^K\in \mathcal{D}_j$ for some $j$ and ${\bf c}_{m-K+1}^m\in \mathcal{D}_k$ for $k-j=m-K+1$ $\text{mod}(d)$. Then, there exists $n_0\geq 0$ such that, for all $n\geq n_0$, 
$\mathbb{M}^{nd+d-(k-j)}({\bf c}_{m-K+1}^m, {\bf c}_1^K)=1$.
We deduce that $\mathbb{P}^{(n+1)d}({\bf c},{\bf c})>0$ for all large enough $n$. Suppose that $d(\mathcal{R}_i)>d$.
By the former affirmation $d(\mathcal{R}_i)$
divides $(n+1)d$ and $(n+2)d$ for large enough $n$ and thus divides $(n+2)d-(n+1)d=d$. But since $d(\mathcal{R}_i)>d$, we obtain a contradiction. $\;\;\square$
\medskip

In our context, the following definition will be useful.
\begin{defi}
A Markov chain $X$ with finite state space is {\it essentially irreducible} if $\mathbb{P}$ has a unique recurrent class. 
\end{defi}
Observe that in the above definition we do not exclude the presence of transient states. Therefore, the notion of essential irreducibility is different from the notion of irreducibility. It is well known that an essentially irreducible chain has a unique invariant probability measure.
As a direct consequence of Theorem \ref{theo1} we have the following 
\begin{coro}
If $\mathbb M$ has a unique closed class then $X$ is essentially irreducible.
\end{coro}


To check that $\mathbb{M}$ has a unique closed class we can use the following simple result.
\begin{prop}
 \label{prop1}
$\mathbb{M}$ has a unique closed class if and only if the matrix $\sum_{n=1}^{|A|^K} \mathbb{M}^n$ has a column with only $1$'s.
 \end{prop}
 \noindent
{\it Proof:} Suppose that $\mathbb{M}$ has at least two closed classes. Since $\sum_{n=1}^{|A|^K} \mathbb{M}^n$ has a column with only 1's, there exists a state ${\bf v}\in A^K$ such that for all ${\bf u}\in A^K$ there exists $n_{\bf u}\in \{ 1,\dots,|A|\}$ such that $\mathbb{M}^{n_{\bf u}}({\bf u}, {\bf v})=1$. Clearly, ${\bf v}$ must belong to a closed class. But we obtain a contradiction because if there are at least two closed classes there must exist some element from which ${\bf v}$ is not accessible.

On the other hand, suppose that $\mathbb{M}$ has a unique closed class and consider ${\bf v}$ in this closed class. For all ${\bf u}\in A^K$, there exists some integer $n_{\bf u}$ such that $\mathbb{M}^{n_{\bf u}}({\bf u}, {\bf v})=1$. Since $A$ is finite, we can always take $n_{\bf u}\leq |A|^K$. This implies that $\sum_{n=1}^{|A|^K} \mathbb{M}^n$ has a column with entries all equal to 1. $\;\;\square$
 \medskip

Regarding the irreducibility of a Markov chain, we have the following

 \begin{prop}
 \label{prop2}
 Suppose that $X$ is irreducible. Then, the only two possible cases are $K=m$ or $K=1$. In the case $K=1$, we have $p({\bf x},a)>0$ for all ${\bf x}\in A^m$ and $a\in A$.
 \end{prop}
 \noindent
{\it Proof:}
 Consider $\mathbb{P}$ irreducible and suppose that $K<m$. Then, $\mathbb{M}$ has the following structure
$$\mathbb{M}({\bf u},{\bf v})
=\left\{
\begin{array}{ll}
	1, &\mbox{whenever}\phantom{*}({\bf u},{\bf v})=({\bf x}_{1}^{K},{\bf x}_{2}^{K}a),\; {\bf x}_{1}^{K}\in A^K \mbox{ and } a\in A,\\
	0, &\mbox{ otherwise}.
\end{array}
\right.
$$
But, if $\mathbb{M}$ has the above structure then $p({\bf x},a)>0$ for all ${\bf x}\in A^m$ and $a\in A$. Therefore we deduce that $K=1$. $\;\;\square$

The contraposition of Proposition \ref{prop2} gives us a simple criterion to check that a given Markov chain is not irreducible: if $m>K>1$ then the chain is not irreducible. When $K = 1$, the chain is clearly irreducible. The only remaining case is $ K=m$, for which we have the following simple criterion.
 
\begin{prop} \label{prop3}
    Let $\mathcal{S}$ be a skeleton of order $K=m$. If there exists ${\bf w}\in \mathcal{S}$ of length $\ell<m$, such that $p({\bf z}{\bf w},a)=0$, for some $a\in A$ and ${\bf z}\in A^{m-\ell}$, then $X$ is not irreducible.
\end{prop}
\begin{proof}
     Since ${\bf w}\in \mathcal{S}$, we obtain that $p({\bf y}{\bf w},a)=0$ for all ${\bf y}\in A^{m-\ell}$. Also, as $|{\bf w}|<m$, we have that $|{\bf w}a|\leq m$, hence $\exists\; {\bf u}\in A^m$ such that ${\bf w}a$ is a suffix of ${\bf u}$. We deduce that for all ${\bf x}\in A^m$, $\mathbb{P}({\bf x},{\bf u})=0$ and therefore $X$ is not irreducible.
\end{proof}


\section{Skeleton determination algorithm}
In this section we present a simple algorithm to obtain the skeleton of a Markov chain. 
\medskip

We consider a Markov chain $X$ of order $m$ with state space $A=\{a_1,\cdots,a_{|A|}\}$ where $|A|$ is the cardinal of $A$. In the following algorithm we will contruct a decreasing sequence of subtrees of the following \emph{rooted tree}
\[
T:=\Big\{v_m\dots v_2v_1\varnothing,\; v_i\in A,\; i\in \{1,\dots,m\}\Big\}
\]
where $\varnothing$ is the root of the tree. Let us now introduce some basic notions for the tree $T$ that will also be used for its subtrees. For $1\leq k\leq m$, an element $v_k\dots v_1$ is a {\it node} of depth $k$. For $k \geq 2$ and a given node $v_k\dots v_1$, its {\it parent node} is the node $v_{k-1}\dots v_1$ (the node $v_{k-2}\dots v_1$ is not considered a parent). The node $v_k\dots v_1$ is called a {\it child} of $v_{k-1}\dots v_1$. Also, we call {\it leaf} a node without any child. For example for the tree $T$ the leaves are all the nodes of depth~$m$. 

We also consider the following probability kernels: for $1\leq k\leq m$, $n\geq k-1$, each node $v_k\dots v_1$ and $a\in A$, let
\[
q_k(v_k\dots v_1,a):=P(X_{n+1}=a\mid X_n=v_1,X_{n-1}=v_2, \dots,X_{n-k+1}=v_k).
\]
From now on, to simplify notation, we will just write $q$ instead of $q_k$.
\begin{remark} The above family of kernels extends the kernel $p$ from Section~2. Nevertheless, we warn the reader that in Section~2 the index $m$ was used for the most recent observation of the Markov chain. Here, we changed this convention to maintain the natural notion of depth in $T$. Still, observe that the node $v_k\dots v_1$ is written with time increasing from left to right: $v_1$ is the most recent observation and $v_k$ the oldest. This is the standard convention for suffix trees and is consistent with the notation used in Section~2, where $p({\bf x},a)$ conditions on ${\bf x}=(x_1,\dots,x_m)$ with $x_m$ being the most recent observation. 
\end{remark}
Consider the parent node $\chi$ of a leaf, the set of its children leaves is 
$$\mathcal{X}=\{\chi_j:\chi_j=a_j\chi,\;j=1,\cdots,|A|\}.$$
For each $\chi_j\in\mathcal{X}$ we define its {\it transition vector} $t(\chi_j)$ as follows: $t(\chi_j):=(t_1,\cdots,t_{|A|})$, where $t_i=\mathbf{1}\{q(\chi_j,a_i)>0\}$ for $i\in \{1,\dots,|A|\}$. 
\medskip

\noindent
The algorithm is as follows:
\medskip

\noindent
\begin{center}
\textbf{Skeleton determination algorithm}
\medskip

\medskip
\begin{minipage}{0.7\textwidth}
\begin{quote}
\textbf{Input:} Tree $T$ of depth $m$ with transition vectors $t(\cdot)$.  \\
\textbf{Output:} Reduced tree isomorphic to the skeleton $\mathcal{S}$.
\end{quote}
\begin{quote}
\begin{tabbing}
\hspace{1.2cm}\=\hspace{1.2cm}\=\kill
Set $\ell \leftarrow m$ and $T_\ell \leftarrow T$.\\
\textbf{while} $\ell \ge 1$ \textbf{do}\\
\> For each set $\mathcal{X} \subset T_\ell$ of children leaves at depth $\ell$:\\
\>\> \textbf{if} $t(\chi_j)=t(\chi_{j'})$ for all $\chi_j,\chi_{j'}\in \mathcal{X}$ \textbf{then}\\
\>\> Remove all nodes in $\mathcal{X}$ from $T_\ell$.\\
\>\> \textbf{end if}\\
\> Set $T_{\ell-1}$ equal to the resulting tree.\\
\> \textbf{if} $T_{\ell-1}=T_\ell$ \textbf{then stop}.\\
\> Set $\ell \leftarrow \ell-1$.\\
\textbf{end while}
\end{tabbing}
\end{quote}
\end{minipage}
\end{center}

\medskip

The tree obtained at the end of the algorithm is isomorphic to the skeleton $\mathcal{S}$ of the original Markov chain. Also, the transition vectors of its leaves provide the \textit{skeleton matrix} $\mathbb{M}$.
\begin{figure}[h!]
\footnotesize
\centering
    \begin{forest}
    for tree={grow=south},
    tikz = {\draw[->, thick, shorten >=70pt, shorten <=95pt] (l) to (r);}
    [,phantom, s sep=0.5cm
        [{\vdots}
            [{w\,[0.5,\,0.5]}, name=l
                [{0w\,[0.6,\,0.4]}
                    [{00w\,[0,\,1]}][{10w\,[0.3,\,0.7]}]]
                [{1w\,[0.25,\,0.75]}
                    [{01w\,[0.4,\,0.6]}][{11w\,[0.95,\,0.05]}]]]]
        [{\vdots}
            [{w\,[1,\,1]}, name = r
                [{0w\,[1,\,1]}
                    [{00w\,[0,\,1]},red][{10w\,[1,\,1]}]]
                [{1w\,[1,\,1]}
                    [{01w\,[1,\,1]}, blue][{11w\,[1,\,1]},blue]]]
    ]]
    \end{forest}
    \caption{Part of the tree $T$ corresponding to the descendants of node w. On the feft-hand figure, for each node we show the corresponding transition probabilities $[q(\cdot,0),q(\cdot,1)]$. On the right-hand figure, we show the transition vectors
    $[t_0,t_1]$ associated to each node. The node $00\text{w}$ (in red) will not be cut, while nodes $01\text{w}$ and $11\text{w}$ (in blue) will be cut off the tree.}
    \label{algo1}
\end{figure}
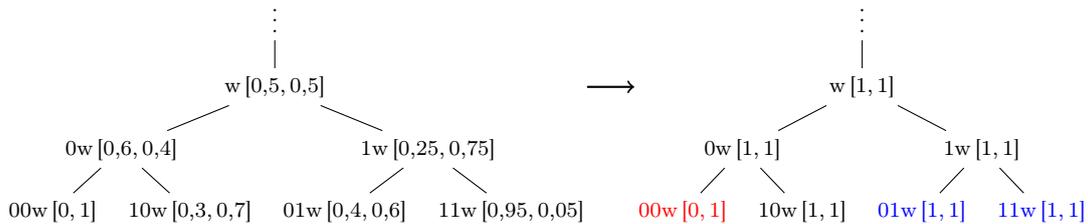
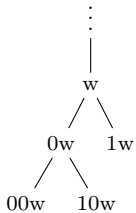
\begin{figure}[h!]
    \footnotesize
    \centering
        \begin{forest}
        for tree={grow=south},
        [{\vdots}
            [{w},
                [{0w}
                    [{00w}][{10w}]]
                [{1w}]]]
        \end{forest}
        \caption{Resulting tree after step 1.}
        \label{algo2}
\end{figure}
\medskip

Figure \ref{algo1} shows step 1 of the algorithm for a portion of a binary tree $T$ ($A=\{0,1\}$), highlighting in blue a set of children leaves that get cut. In red is represented the leaf 00w whose context prohibits a transition unlike the leaf 10w. Therefore, the algorithm will not cut any of them. The resulting tree after step 1 is presented in Figure~\ref{algo2}.

\section{Illustrative example}
We now present an example of a Markov chain and its skeleton that shows that, in some cases, $K$ can be significantly smaller that $m$. 

We consider a Markov chain $X$ on state space $A=\{0,1\}$ with associated context tree $\mathfrak{T}$ given in Figure \ref{tree} (see for example \cite{cenac} for the definition of a context tree). Two of its leaves force a specific transition:
$$q(10,0)=1,\;\;q(111,0)=1.$$
For all the other leaves w $\in \mathfrak{T} \setminus \{10,\,111\}$, $q(\text{w},a)\in(0,1),\forall a\in A$. Therefore, we deduce that $X$ is a Markov chain of order $m=10$. 

Applying the skeleton determination algorithm given in the former section, we obtain a skeleton of order 3, shown in Figure \ref{skel}. Its skeleton matrix $\mathbb{M}$ is shown in Figure \ref{matrix}. Looking at $\mathbb{M}$, we can easily check that $X$ is essentially irreducible  and therefore has a unique invariant probability measure. Since $1<K<m$, by the contraposition of Proposition \ref{prop2}, we also know that $X$ is not irreducible. The original transition matrix of $X$ is too large and sparse to produce a proper visualization ($2^{10}=1024$ rows and columns, with a maximum of 2 positive entries per row).

\begin{figure}[h!]
    \begin{minipage}{0.5\textwidth}
        \tiny{
        \begin{forest}
        for tree={grow=west}
        [$\varnothing$,  
         [0,  
          [0,  
           [0[0[0[0[0][1[0][1[0][1[0][1]]]]][1]][1[0][1[0[0][1[0][1]]][1]]]][1[0[0][1]][1[0[0[0][1]][1]][1[0][1[0[0[0][1]][1]][1]]]]]]
           [1[0[0[0[0][1[0[0][1]][1]]][1[0[0[0[0][1]][1]][1[0][1]]][1]]][1]][1]]]
          [1, circle, draw]
         ]
         [1,  
          [0,  
           [0[0[0[0][1[0[0[0[0][1]][1]][1[0][1]]][1]]][1]][1[0[0[0][1[0][1]]][1[0[0][1]][1[0][1[0[0][1]][1]]]]][1[0][1[0][1[0][1[0][1]]]]]]]
           [1[0[0][1]][1]]
          ]
          [1,  
           [0,  
            [0[0[0[0][1[0][1]]][1]][1[0][1[0[0[0[0][1]][1]][1]][1[0[0][1]][1]]]]]
            [1[0[0][1[0][1[0[0[0][1]][1]][1[0][1[0][1]]]]]][1[0][1[0][1]]]]]
           [1, circle, draw]
          ]
         ]
        ]
        \end{forest}
        }
        \caption{Context tree $\mathfrak{T}$ of the 10-th order Markov chain $X$. The circled nodes correspond to prohibited transitions toward $1$. All the other contexts gives positive probabilities toward $0$ and $1$.}
        \label{tree}
    \end{minipage}
    \begin{minipage}{0.48\textwidth}
    
    \centering
        \begin{forest}
        for tree={grow=west}
        [$\varnothing$,  
         [0,  
          [0, ]
          [1, circle, draw]
         ]
         [1,  
          [0,]
          [1,  
           [0, ]
           [1,  circle, draw]
          ]
         ]
        ]
        \end{forest}
        \caption{Skeleton $\mathcal{S}$ of $X$, contexts that prohibit transitions are circled.}
        \label{skel}
        \includegraphics[width=\textwidth]{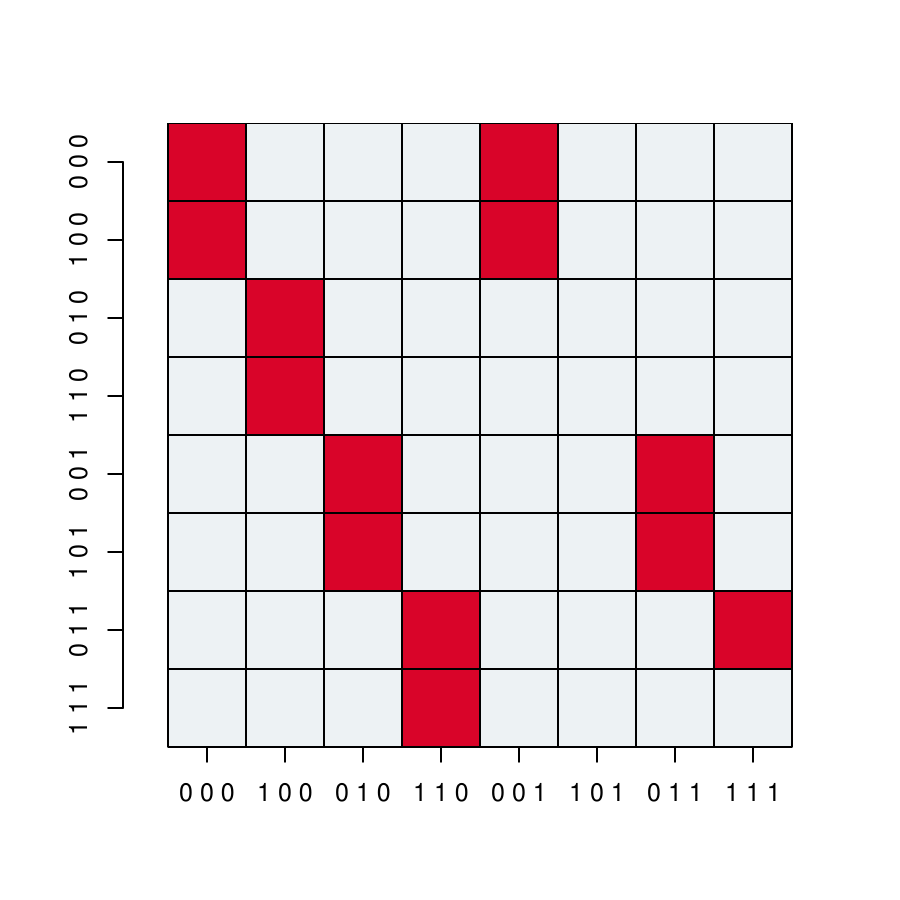}
        \caption{Skeleton matrix $\mathbb{M}$ of $\mathcal{S}$: $1$'s in red and $0$'s in grey.}
        \label{matrix}
        \vspace{2.5cm}
    \end{minipage}
\end{figure}

\section{Computational cost}

We now examine the computational cost to verify the essential irreducibility of higher-order Markov chains. For this, we state without proof the following classical criterion for essential irreducibility.
\begin{prop}
 \label{prop4}
The Markov chain $X$ with transition matrix $\mathbb{P}$ is essentially irreducible if and only if the matrix $\sum_{n=1}^{|A|^m} \mathbb{P}^n$ has a column with only strictly positive entries.
 \end{prop}
 \noindent

Using a naive multiplication approach to test the essential irreducibility of a first-order Markov chain with alphabet $D$, one may require up to $|D|-1$ matrix powers, each involving $|D|^3$ operations. We then have to sum all the terms, leading to an overall computational complexity of $\mathcal{O}(|D|^4)$. Consequently, after reducing an $m$-th order Markov chain to a first-order chain on the enlarged alphabet $A^m$, testing essential irreducibility using Proposition \ref{prop4} can require up to $\mathcal{O}(|A|^{4m})$ operations. 

Alternatively, computationally efficient graph-based methods such as Tarjan’s algorithm may be used to verify the uniqueness of the recurrent class, \emph{i.e.}, essential irreducibility, directly from the transition graph \cite{tarjan1972depth}. Applied to the reduced first-order chain, this approach has complexity $\mathcal{O}(|A|^{m+1})$. 

By contrast, if the associated skeleton has order $K < m$, essential irreducibility can be assessed by simply computing the skeleton and then applying Tarjan's algorithm to the matrix $\mathbb{M}$. This procedure has computational cost $\mathcal{O}(|A|^m)$. When $K \leq m/4$, even if we naively use Proposition \ref{prop1}, we still obtain a computational cost $\mathcal{O}(|A|^m)$.  This shows that, in general, using the skeleton determination algorithm can reduce the computational complexity of checking the essential irreducibility of a high-order Markov chain.

\begin{remark}
Our algorithm works for general $m$-th order Markov chains, but when the Markov chain has a special structure, simple alternative ways to classify it may exist. This is the case for Mixture Transition Distribution models that are popular models for sparse $m$-th order Markov chains \cite{berchtold2002mixture, ost2023sparse}. These are models in which the transition kernel can be written as follows. For all ${\bf x} \in A^m$ and $a \in A$
\begin{equation*}
p({\bf x},a) = \lambda_0p_0(a)+\sum_{i = 1}^m \lambda_ip_i(x_{m+1-i}, a),
\end{equation*}
where $p_0$ is a probability on $A$, for all $i\geq 1$, $p_i:A \times A \to [0,1]$ is a transition kernel for a Markov chain of order 1 and, for all $i\geq 0$, $\lambda_i \geq 0$ with $\sum_{i=0}^m\lambda_i = 1$. An interesting case is when $\lambda_0 = 0$. In this case, let $X$ be a Markov chain with transition kernel $p$. If the stochastic matrix $\tilde{P}$, defined by $\tilde{P}(a,b) = \sum_{i = 1}^m \lambda_ip_i(a,b)$ for $(a,b)\in A^2$, is irreducible and aperiodic, then $X$ is essentially irreducible 
(see \cite{de2015one}, where a class of infinite order chains that generalize MTD models is studied). This criterion allows us to efficiently verify essential irreducibility for MTD models without constructing $\mathbb{P}$, if the stochastic matrices $p_i$ are known.
\end{remark}



\section*{ Acknowledgements}
C.G. was partially supported by FAPESP Grant 2023/07228-9.   C.T.G.H.O. thanks PIBIC/CNPq. D.Y.T. was partially supported by CNPq Grant 421955/2023-6, CAPES Grant 88887.627882/2021-00, Serrapilheira Grant 2023.

\end{document}